\documentclass[12pt,reqno]{amsart}

\usepackage[T1]{fontenc}
\usepackage{ae,aecompl}
\usepackage{graphics,cancel,graphicx}
\usepackage{epsfig,psfrag,manfnt}
\usepackage{mathrsfs}
\usepackage{graphicx,mathabx}
\usepackage{bbm,subfigure,rotating,bm,float,mathdots,wasysym,scalerel}

\usepackage{stmaryrd}
\usepackage[all,cmtip]{xy}
\usepackage{amssymb,amsmath,amsfonts,amsthm,color}

\usepackage{scalerel,stackengine}
\stackMath
\newcommand\reallywidehat[1]{%
\savestack{\tmpbox}{\stretchto{%
  \scaleto{%
    \scalerel*[\widthof{\ensuremath{#1}}]{\kern-.6pt\bigwedge\kern-.6pt}%
    {\rule[-\textheight/2]{1ex}{\textheight}}%WIDTH-LIMITED BIG WEDGE
  }{\textheight}% 
}{0.5ex}}%
\stackon[1pt]{#1}{\tmpbox}%
}
\newcommand{\calA}{\mathcal{A}}

\newcommand{\calL}{\mathcal{L}}

\newcommand{\calR}{\mathcal{R}}

\newcommand{\calU}{\mathcal{U}}
\newcommand{\calV}{\mathcal{V}}

\newcommand{\mC}{\mathbb{C}}
\newcommand{\mD}{\mathbb{D}}

\newcommand{\mF}{\mathbb{F}}

\newcommand{\mN}{\mathbb{N}}

\newcommand{\mR}{\mathbb{R}}

\newcommand{\bbe}{\mathbf{e}}

\newcommand{\bbn}{\mathbf{n}}

\newcommand{\bbv}{\mathbf{v}}
\newcommand{\bbw}{\mathbf{w}}

\newcommand{\bbz}{\mathbf{z}}

\newcommand{\bbeta}{\bm{\eta}}
\newcommand{\bxi}{\bm{\xi}}
\newcommand{\bzeta}{\bm{\zeta}}

\newcommand{\SL}{{\textrm{SL}}}
\newcommand{\eSL}{{\textrm{\em SL}}}
\newcommand{\E}{{\textrm{E}}}
\newcommand{\eE}{{\textrm{\em E}}}

\newtheorem{theorem}{Theorem}[section]
\newtheorem{lemma}[theorem]{Lemma}
\newtheorem{corollary}[theorem]{Corollary}
\newtheorem{proposition}[theorem]{Proposition}

\theoremstyle{definition}

\theoremstyle{definition}
\newtheorem{definition}[theorem]{Definition}

\theoremstyle{definition}

\theoremstyle{definition}

\begin{document}



 \title[]{Banach algebras of symmetric functions on the polydisc}
 
 \author[]{Amol Sasane}
\address{Department of Mathematics \\London School of Economics\\
     Houghton Street\\ London WC2A 2AE\\ United Kingdom}
\email{A.J.Sasane@lse.ac.uk}
 
\maketitle
 
 \vspace{-0.69cm} 
  
\begin{abstract}
Let $\mD=\{z\in \mC:|z|<1\}$ and for an integer $d\geq 1$, let $S_d$ denote the symmetric group, consisting of of all permutations of the set $\{1,\cdots, d\}$. A function  $f:\mD^d\rightarrow \mC$ is symmetric if 
$f(z_1,\cdots, z_d)=f(z_{\sigma(1)},\cdots, z_{\sigma (d)})$ for all $\sigma \in S_d$ and all $(z_1,\cdots, z_d)\in \mD^d$. The polydisc algebra $A(\mD^d)$ is the Banach algebra of all holomorphic functions $f$ on the polydisc $\mD^d$ that can be continuously extended to the closure of the polydisc in $\mC^d$, with pointwise operations and the supremum norm (given by $\|f\|_\infty:=\sup_{\bbz \in \mD^d} |f(\bbz)|$). 
Let $A_{\textrm{sym}}(\mD^d)$ be the 
Banach subalgebra of $A(\mD^d)$ consisting of all symmetric functions in the polydisc algebra. Algebraic-analytic properties of $A_{\textrm{sym}}(\mD^d)$ are investigated. In particular, the following results are shown: the corona theorem, description of the maximal ideal space and its contractibility,  Hermiteness, projective-freeness, and non-coherence. 
\end{abstract}
 
 \section{Introduction}
 
 \noindent Symmetric polynomials and the ring of symmetric functions are interesting objects in algebra. 
 We study an analogous object from the Banach algebra viewpoint. The aim of this article is to study the algebraic-analytic properties of Banach algebras of `symmetric functions' (defined below) on the polydisc.
 
 \begin{definition}
 $\;$
 
 \noindent $\bullet$  Let $\mD:=\{z\in \mC: |z|<1\}$.  Let $\overline{\mD}$ denote the closure of $\mD$ in $\mC$. 
 
 \smallskip 
 
 \noindent  $\bullet$ Let $\mN:=\{1,2,3,\cdots\}$. For $d\in \mN$, let $A(\mD^d)$ denote the {\em polydisc}
 \phantom{$\bullet$ }{\em algebra} of all functions $f:\overline{\mD}^d\rightarrow \mC$ that are holomorphic in $\mD^d$ and  \phantom{$\bullet$ }continuous on $\overline{\mD}^d$. 
 Define 
 $\displaystyle 
\|f\|_\infty:=\sup_{\bbz \in \mD^d}|f(\bbz)|$ for $ f\in A(\mD^d).
$ 

\noindent \phantom{$\bullet$ }Then $A(\mD^d)$ is  a Banach algebra with pointwise operations and the \phantom{$\bullet$ }{\em supremum norm} $\|\cdot\|_\infty$. 
 
 \smallskip 
 
\noindent   $\bullet$ 
Let $d\in \mN$. Then $S_d$ denotes the {\em symmetric group}, 
consisting of all \phantom{$\bullet$ }bijective maps $\sigma:\{1,\cdots, d\}\rightarrow \{1,\cdots, d\}$, with 
composition $\circ$ taken \phantom{$\bullet$ }as the group operation.  The group $S_d$ acts on $\mC^d$ as follows. 

\noindent 
\phantom{$\bullet$ }If $\bbz=(z_1,\cdots, z_d)\in \mC^d$ and $\sigma \in S_d$, then we define 
$\sigma\;\!\bbz \in \mC^d$ by 
$$
\sigma \;\!\bbz:=(z_{{\scaleobj{0.75}{\sigma(1)}}},\cdots, z_{{\scaleobj{0.75}{\sigma(d)}}}).
$$
\phantom{$\bullet$ }A function $f\in A(\mD^d)$ is {\em symmetric} if for all $\bbz\in \mD^d$, and all $\sigma \in S_d$, $\;$\phantom{$\bullet$ }we have $f(\bbz)=f(\sigma\;\! \bbz)=:(\sigma f)(\bbz)$. 

\smallskip 

\noindent $\bullet$ Define the {\em polydisc algebra of symmetric functions} by 
$$
A_{\textrm{sym}}(\mD^d):=\{f\in A(\mD^d): f\textrm{ is symmetric}\}.
$$
\end{definition}

\noindent Then $A_{\textrm{sym}}(\mD^d)$ is a Banach subalgebra of $A(\mD^d)$ with the same operations and the same norm. Closure under addition, scalar multiplication, and multiplication is clear. 
Also, if $(f_n)_{n\in \mN}$ is a Cauchy sequence in $A_{\textrm{sym}}(\mD^d)$, then being Cauchy in $A(\mD^d)$ too, it converges to some $f\in A(\mD)$, and as uniform convergence implies pointwise convergence, 
$$
f(\sigma\;\! \bbz)= \lim_{n\rightarrow \infty} f_n(\sigma\;\! \bbz)
=\lim_{n\rightarrow \infty} f_n(\bbz)
=f(\bbz) \quad (\bbz\in \mD^d).
$$
In this article, we focus on the algebraic-analytic properties of the symmetric polydisc algebra 
as a concrete prototypical example. However, more generally, for a `symmetric' domain $\Omega\subset \mC^d$, one could also study the properties of appropriately defined `symmetric function algebras' such as $A_{\textrm{sym}}(\Omega)$, $ H^\infty_{\textrm{sym}}(\Omega)$ (symmetric subalgebra of the Hardy algebra $H^\infty(\Omega)$ of bounded and holomorphic functions on $\Omega$), etc.

\section{Corona problem}

\noindent Let $\mathbf{1}\in A_{\textrm{sym}}(\mD^d)$ be the constant function taking value $1$ everywhere. 

\begin{theorem}
\label{theorem_25_dec_2021_19:52}
Let $n\in \mN$ and $f_1,\cdots, f_n\in A_{\textrm{\em sym}}(\mD^d)$. 

\noindent Then the following are equivalent:

\noindent {\em (1)} There exist $g_1,\cdots, g_n\in A_{\textrm{\em sym}}(\mD^d)$ such that $f_1g_1+\cdots+f_ng_n=\mathbf{1}$. 

\noindent {\em (2)} There exists a $\delta>0$ such that $\forall \bbz\in \mD^d$, $|f_1(\bbz)|+\cdots+|f_n(\bbz)|\geq \delta$.
\end{theorem}
\begin{proof} (1)$\Rightarrow$(2) follows from the triangle inequality: For all $\bbz\in \mD^d$, 
\begin{eqnarray*}
0<1=|\mathbf{1}(\bbz)|\!\!&=&\!\!|f_1(\bbz)g_1(\bbz)+\cdots+f_n(\bbz)g_n(\bbz)|\\
\!\!&\leq&\!\! ( |f_1(\bbz)|+\cdots+|f_n(\bbz)|) \max_{1\leq i\leq n} \|g_i\|_\infty,
\end{eqnarray*}
and so we may take $\delta:=\displaystyle \Big(\max_{1\leq i\leq n} \|g_i\|_\infty\Big)^{-1}>0$. 

\smallskip

\noindent 
The converse (2)$\Rightarrow$(1) follows from the corona theorem for $A(\mD^d)$, and by `symmetrising'  the solution as follows. The corona condition (2) gives the existence of $g_1,\cdots, g_n\in A(\mD^n)$ such that 
 $
f_1g_1+\cdots+f_ng_n=\mathbf{1}.
$ 
For any permutation $\sigma \in S_d$, we have for all $\bbz=(z_1,\cdots, z_d)\in \mD^d$ that 
\begin{eqnarray*}
1\!\!\!&=&\!\!\!\sum_{1\leq i \leq n} 
f_i(z_{{\scaleobj{0.75}{\sigma(1)}}},\cdots, z_{{\scaleobj{0.75}{\sigma(d)}}})\;\!g_i(z_{{\scaleobj{0.75}{\sigma(1)}}},\cdots, z_{{\scaleobj{0.75}{\sigma(d)}}})\\
\!\!\!&=&\!\!\!
\sum_{1\leq i \leq n} f_i(z_{{\scaleobj{0.75}{1}}}, \cdots, z_{{\scaleobj{0.75}{d}}}) \;\!g_i(z_{{\scaleobj{0.75}{\sigma(1)}}},\cdots, z_{{\scaleobj{0.75}{\sigma(d)}}}).
\end{eqnarray*}
Summing the above over all $\sigma \in S_d$, we obtain 
  $
 f_1\widetilde{g}_1+\cdots+f_n \widetilde{g}_n=\mathbf{1},
 $ 
 where $\widetilde{g}_1,\cdots, \widetilde{g}_n\in A_{\textrm{sym}}(\mD^d)$ are defined by
$$
\phantom{AAa}
\widetilde{g}_i(\bbz) := \frac{1}{d!} \sum_{\sigma\in S_d} g_i(z_{{\scaleobj{0.75}{\sigma(1)}}},\cdots, z_{{\scaleobj{0.75}{\sigma(d)}}}), \quad 
\bbz=(z_1,\cdots, z_d)\in \mD^d.
\phantom{AAa}\qedhere
$$
\end{proof}

\section{Maximal ideal space of $A_{\mathrm{sym}}(\mD^d)$}

\noindent Let $\sim$ be the relation on $\overline{\mD}^d$ defined as follows:  
 $\bbz \sim\!\bbw$ if there exists a permutation $\sigma \in S_d$ such that 
$$
\sigma\;\! \bbz= \bbw.
$$
Then $\sim$ is an equivalence relation on $\overline{\mD}^d$. 
For $\bbz\in \overline{\mD}^d$, the {\em orbit} of $\bbz$ is 
$$
S_d\;\! \bbz:=\{\sigma\;\! \bbz: \sigma\in S_d\}, \quad 
$$  
The equivalence class of $\bbz \in \overline{\mD}^d$ is denoted by 
$$
[\bbz]=S_d \;\!\bbz.
$$
 Let $\overline{\mD}^d/\!\sim$ denote the set of all equivalence classes of $\overline{\mD}^d$ under the equivalence relation $\sim$. The natural projection map 
$$
\begin{array}{rccc} 
\pi:& \overline{\mD}^d & \rightarrow & \overline{\mD}^d/\!\sim\\[0.1cm]
& \bbz & \mapsto &[\bbz] 
\end{array}
$$
is surjective. We use the subspace topology on $\overline{\mD}^d$ induced from the usual Euclidean topology on $\mC^d$. 
We endow the set $\overline{\mD}^d/\!\sim$ with the quotient topology $\tau_{\textrm{quot}}$, 
\phantom{$\overline{\mD}^d$}\!\!\!\!\!\!\!that is, the finest/strongest topology making the above projection map $\pi:\overline{\mD}^d  \rightarrow  \overline{\mD}^d/\!\sim$ continuous. As the topological space $\overline{\mD}^d$ is compact, it follows that the quotient space $(\overline{\mD}^d/\!\sim, \tau_{\textrm{quot}})$ is compact too. 
For any set $X\subset \overline{\mD}^d$ and any $\sigma \in S_d$, 
define 
$$
\sigma X:=\{ \sigma \;\!\bxi :\bxi \in X\}.
$$

\begin{proposition}
\label{proposition_27_dec_2021_16:18}
$(\overline{\mD}^d/\!\sim,\tau_{\textrm{\em quot}})$ is Hausdorff.
\end{proposition}
\begin{proof} Let $[\bbz],[\bbw]\in \overline{\mD}^d/\!\sim$ (where $\bbz,\bbw\in \overline{\mD}^d$) be such  that $[\bbz]\neq [\bbw]$. Then $(S_d \;\!\bbz)\cap (S_d \;\!\bbw)=\emptyset$. \phantom{$\overline{\mD}^d$}\!\!\!\!\!\!\!
But  $S_d\;\! \bbw$ is a finite set (with $d!$ elements). As $\overline{\mD}^d$ with its usual Euclidean topology is Hausdorff, it follows that there 
exist  open sets $Z,W$ in $\overline{\mD}^d$ such that $S_d\;\!\bbz\subset Z$, $S_d \;\!\bbw \subset W$, and $Z\cap W=\emptyset$. 
For any $\sigma \in S_d$, the map $\bzeta \mapsto \sigma \bzeta:\overline{\mD}^d\rightarrow \overline{\mD}^d$ is a homeomorphism, and so $\sigma Z=\{\sigma\bzeta:\bzeta \in Z\}$ is open 
in $\overline{\mD}^d$. 
Define 
$$
U=\bigcap_{\sigma\in S_d} \sigma Z\;\;\textrm{ and }\;\;
V:=\bigcap_{\sigma\in S_d} \sigma W.
$$
Then $U,V$ are open in $\overline{\mD}^d$. Moreover, for all $\sigma \in S_d$, $\sigma\;\! U=U$ and $\sigma V=V$. For any $\sigma \in S_d$, we have $\sigma^{-1} \bbz \in Z$, and so $\bbz \in \sigma Z$, showing that $\bbz \in U$. Similarly, $\bbw \in V$. 
Define the sets 
\begin{eqnarray*}
\calU\!\!\!&:=&\!\!\!\{[\bzeta]:\bzeta\in U\}\;\;\;\! (\owns [\bbz]),\textrm{ and}\\
\calV\!\!\!&:=&\!\!\!\{[\bbeta]:\bbeta\in V\}\;\;  (\owns [\bbw]).
\end{eqnarray*}
We claim that $\pi^{-1}\calU=U$. Firstly, if $\bzeta \in U$, then $[\bzeta]\in \calU$, that is, $\pi \bzeta \in \calU$, and so $\bzeta \in \pi^{-1}\calU$, showing that $U\subset \pi^{-1} \calU$. Secondly, if $\bzeta \in \pi^{-1}\calU$, then $\pi \;\!\bzeta \in \calU$, so that $[\bzeta]=\pi \;\!\bzeta=[\bbeta]$ for some $\bbeta \in U$, giving, for some $\sigma \in S_d$, that $\bzeta=\sigma \bbeta \in \sigma U=U$, showing that $\pi^{-1}\calU\subset U$ too. 

 Since $\pi^{-1} \calU=U$ and $\pi^{-1}\calV=V$ are open in $\overline{\mD}^d$, we conclude that  the sets  $\calU, \calV$ are open in $(\overline{\mD}^d/\!\sim,\tau_{\textrm{quot}})$.  
 
 Finally, we show that $\calU\cap \calV=\emptyset$. Suppose not. Then there exist $\bzeta\in U$ and $ \bbeta \in V$  such that $[\bzeta]=[\bbeta]$, that is, $\bzeta=\sigma \;\!\bbeta $ for some $\sigma \in S_d$. But $\bbeta \in U\subset Z$, and 
 $\sigma \;\!\bbeta \in \sigma V=V \subset W$, so that $\bzeta =\sigma \;\!\bbeta \in Z\cap W=\emptyset$, a contradiction. 
  \end{proof}

\noindent The {\em maximal ideal space} of a complex unital commutative Banach algebra $\calA$ is denoted by $M_\calA$, and it is the set of all nontrivial complex homomorphisms $\varphi:\calA\rightarrow \mC$. 
Let $\calL(\calA;\mC)$ be dual space of $\calA$, that is, the set of all continuous linear maps $\varphi:\calA\rightarrow \mC$. We equip $\calL(\calA;\mC)$ with the weak-$\ast$ topology (see e.g. \cite[\S3.14]{Rud}). 
Clearly $M_\calA\subset \calL(\calA;\mC)$. The {\em Gelfand topology} $\tau_{{\scaleobj{0.75}{\textrm{Gel}}}}$ on $M_\calA$ is the topology induced on $M_\calA$ from the weak-$\ast$ topology of $\calL(\calA;\mC)$. 
As the Gelfand topology need not be metrisable, we use the language of nets when discussing limits and convergence in the proof of Theorem~\ref{theorem_28_dec_2021_11:30}\footnote{But after proving this result, we see that the Gelfand topology of $M_{\scaleobj{0.81}{A_{\textrm{sym}}(\mD^d)}}$ is the quotient topology of ${\scaleobj{0.9}{\overline{\mD}^d/\!\sim}}$, and the latter can be seen to be metrisable (since the space ${\scaleobj{0.9}{\overline{\mD}^d}}$ is compact and metrisable, and the quotient space ${\scaleobj{0.9}{\overline{\mD}^d/\!\sim}}$ is Hausdorff; see \cite[\S23, 23K]{Wil}).\phantom{$\overline{\mD}^d/\!\sim$}}. In particular, we will use the notion of subnets in the sense of Willard \cite[\S11]{Wil} (see also \cite[\S1.3.B]{MorRup}). If $(x_i)_{i\in I}$ (where $I$ is a directed set) is a net in a topological space $X$, then a {\em subnet of} $(x_i)_{i\in I}$  
is a net $(x_{h(j)})_{j\in J}$, where $J$ is a directed set, and $h:J\rightarrow I$ is an {\em increasing cofinal function}, that is, it satisfies 

$\bullet$ (increasing) $h(j_1) \leq h(j_2)$ whenever $j_1\leq j_2$, and 

$\bullet$ (cofinal) for each $i\in I$, there is some $j\in J$ such that $i\leq h(j)$. 

\noindent This notion of a subnet affords us the following facts, which we shall need later (see \cite[\S11, \S17]{Wil}): 

$\bullet$ A subnet of a net convergent to $x$ also converges to $x$. 

$\bullet$ A net in a compact set $K\subset X$ possesses a subnet which converges 

\phantom{$\bullet$} to a point in $K$. 

\noindent We will now show that $(M_{A_{\textrm{sym}}(\mD^d)},
 \tau_{{\scaleobj{0.75}{\textrm{Gel}}}})$ can be identified as a topological space with $(\overline{\mD}^d/\!\sim,\tau_{\textrm{quot}})$. To do this, we will use the following:

\begin{lemma} 
\label{lemma_25_dec_2021_20:17}
Suppose that $\varphi, \psi \in M_{A_{\textrm{\em sym}}(\mD^d)}$. 

\noindent Then 
  $\varphi=\psi$ if and only if $\ker \varphi=\ker \psi$. 
  \end{lemma}
  \begin{proof}
  The `only if' part is trivial. `If' part: Let $\ker \varphi=\ker \psi$. Suppose $f\in A_{\textrm{sym}}(\mD^d)$ is such that $\varphi(f)\neq \psi (f)$. Set 
  $  g:=f-\psi(f) \mathbf{1}\in A_{\textrm{sym}}(\mD^d)$. Clearly $g\in \ker \psi$. As $\ker \varphi=\ker \psi$,  $\varphi(g)=0$ too, that is, $\varphi(f)= \psi (f)$, a contradiction. 
  \end{proof}
  
  \begin{lemma}
  \label{lemma_27_dec_2021_15:43}$\;$ 
  
  \noindent 
  If $\tau_1,\tau_2$ are topologies on a set $X$ such that 
  
  $\bullet$ $\tau_1\subset \tau_2$, 
  
  $\bullet$  $\tau_1$ is a Hausdorff topology, and 
  
  $\bullet$  $\tau_2$ is compact, 
  
  \noindent then $\tau_1=\tau_2$.
  \end{lemma}
  \begin{proof}
  This is precisely \cite[Theorem 1.53]{MorRup} or \cite[\S3.8(a)]{Rud}. 
  \end{proof}
  
\begin{theorem}
\label{theorem_28_dec_2021_11:30}$\;$ 

\noindent 
The maximal ideal space of $A_{\textrm{\em sym}}(\mD^d)$ is homeomorphic to $\overline{\mD}^d/\!\sim$. 
\end{theorem}
\begin{proof}  {\bf The map $\iota:\overline{\mD}^d/\!\sim \;\! \rightarrow  M_{A_{\textrm{sym}}(\mD^d)}$:}  

\noindent For $[\bbz] \in \overline{\mD}^d/\!\sim$ (where $\bbz \in \overline{\mD}^d$), 
 we define $\varphi_{[\bbz]}:A_{\textrm{sym}}(\mD^d)\rightarrow \mC$ to be the point evaluation at $\bbz$, that is, 
 $$
 \varphi_{[\bbz]}(f)=f(\bbz), \quad f\in A_{\textrm{sym}}(\mD^d).
 $$
 Then $\varphi_{[\bbz]}$ is well-defined (since $A_{\textrm{sym}}(\mD^d)$ consists of {\em symmetric} functions), and defines a complex homomorphism. Also, $\varphi_{[\bbz]}(\mathbf{1})=1$. 
We show below that the map 
$$
\begin{array}{rccc} 
\iota: &\overline{\mD}^d/\!\sim & \rightarrow & M_{A_{\textrm{sym}}(\mD^d)} \\[0.1cm]
 & [\bbz]& \mapsto & \varphi_{[\bbz]} 
 \end{array}
 $$
 is a homeomorphism. 
 
 \goodbreak
 
 \noindent {\bf The map $\iota$ is injective:} Suppose that 
 $$
 [\bbz]\neq [\bbw]
 $$
  for some $\bbz,\bbw\in \overline{\mD}^d$. Then for all $\sigma \in S_d$, we have $\sigma\;\! \bbz \neq \bbw$. Thus the polynomials $p_\bbz,p_\bbw\in \mC[X]$ defined by 
\begin{eqnarray*}
p_\bbz(X)\!\!\! &:=&\!\!\!(X-z_1)\cdots (X-z_d), \\
p_\bbw(X)\!\!\!&:=&\!\!\!(X-w_1)\cdots(X-w_d),
\end{eqnarray*}
 are distinct. As the coefficients of a polynomial are symmetric polynomials of the zeroes of the polynomial, there exists a symmetric polynomial $f$ such that $f(\bbz)\neq f(\bbw)$, that is, 
 $$
 \varphi_{[\bbz]}(f)\neq \varphi_{[\bbw]}(f).
 $$ 
 
 \smallskip

 \noindent 
 {\bf The map $\iota$ is surjective:}   Let $\varphi\in M_{A_{\textrm{sym}}(\mD^d)}$ be such that 
  $$
\textrm{ for all } \bbz \in \overline{\mD}^d, \;\varphi \neq \varphi_{[\bbz]}.
 $$
 Let $\mathfrak{m}$ be the maximal ideal in $A_{\textrm{sym}}(\mD^d)$ given by $\mathfrak{m}:=\ker \varphi$. 
 
 \noindent 
  By Lemma~\ref{lemma_25_dec_2021_20:17},  
  $$
  (\mathfrak{m}=)\;\ker \varphi \neq \ker \varphi_{[\bbz]}\textrm{ for all }\bbz \in \overline{\mD}^d.
  $$
   Thus for each $\bbz \in  \overline{\mD}^d$, 
 there exists an $f_\bbz\in A_{\textrm{sym}}(\mD^d)$ such that 
 $$
 f_\bbz \in \mathfrak{m},\textrm{ but }f_\bbz(\bbz)\neq 0.
 $$ 
 By the compactness of $\overline{\mD}^d$,  there is an $n\in \mN$ and 
  $f_1,\cdots, f_n\in A_{\textrm{sym}}(\mD^d)$ such that 
  $$
  \forall \bbz \in \overline{\mD}^d, 
 \;\exists k\in \{1,\cdots, n\}\textrm{ such that } 
 f_k(\bbz)\neq 0.
 $$
  Hence there exists a $\delta>0$ such that 
 for all $\bbz \in \mD^d$, 
 $$
 |f_1(\bbz)|+\cdots +|f_n(\bbz)|\geq \delta.
 $$
  By the corona theorem for $A_{\textrm{sym}}(\mD^d)$ (Theorem~\ref{theorem_25_dec_2021_19:52} above),  
 there exist $g_1,\cdots, g_n\in A_{\textrm{sym}}(\mD^d)$ such that 
  $
 f_1g_1+\cdots+f_ng_n=\mathbf{1}.
 $  
 But as the element $f_1g_1+\cdots +f_ng_n$ belongs to the ideal $\mathfrak{m}$, we get $\mathbf{1}\in \mathfrak{m}$, contradicting the maximality of $\mathfrak{m}$. 
 
 \smallskip 
 
 \noindent 
 {\bf The map $\iota$ is a homeomorphism:} 
 
 \noindent 
 Via the bijection $\iota\!:\!\overline{\mD}^d\!/\!\!\sim\;\!\rightarrow  \!M_{A_{\textrm{sym}}(\mD^d)}$, we identify $M_{A_{\textrm{sym}}(\mD^d)}$ and $\overline{\mD}^d\!/\!\sim$. 
 
 \noindent Then showing $\iota$ is a homeomorphism is tantamount to showing that  the topologies $\tau_{\textrm{quot}}$ and $ \tau_{{\scaleobj{0.75}{\textrm{Gel}}}}$ coincide, in view of the above identification. 
To show the coincidence of these two topologies, will use Lemma~\ref{lemma_27_dec_2021_15:43}, 
with $\tau_1= \tau_{{\scaleobj{0.75}{\textrm{Gel}}}}$ and $\tau_2=\tau_{\textrm{quot}}$. 
We want to show that $ \tau_{{\scaleobj{0.75}{\textrm{Gel}}}}\subset \tau_{\textrm{quot}}$. It is enough to show that if $F$ is closed in $ \tau_{{\scaleobj{0.75}{\textrm{Gel}}}}$, then it is closed in $\tau_{\textrm{quot}}$. Suppose that $F$ is closed in 
$ \tau_{{\scaleobj{0.75}{\textrm{Gel}}}}$. Let  $([\bbz_i])_{i\in I}$ (where $I$ is a directed set) be a net in $F$ that is convergent with limit $[\bbz]$  in $(\overline{\mD}^d /\!\sim,\tau_{\textrm{quot}}) $. As $(\bbz_i)_{i\in I}$ is a net in the compact set $\overline{\mD}^d$, it has a convergent subnet $(\bbz_{h(j)})_{j\in J}$ (where $J$ is a directed set and $h:J\rightarrow I$
\phantom{$\overline{\mD}^d_{h(j)})_{j\in J}$}\!\!\!\!\!\!\!\!\!\!\!\!\!\!\!\!\!\!\!\!\!\!\!
 is a monotone final function), convergent with a limit, say $\bbw$ in $\overline{\mD}^d$ (in the Euclidean topology).
 Then the net $([\bbz_{h(i)}])_{j\in J}$ converges to $[\bbw]$ in $(\overline{\mD}^d /\!\sim,\tau_{\textrm{quot}})$. But as $([\bbz_{h(i)}])_{j\in J}$ is a subnet of the convergent net $([\bbz_i])_{i\in I}$
  with limit $[\bbz]$ in $(\overline{\mD}^d /\!\sim,\tau_{\textrm{quot}})$, and since $(\overline{\mD}^d /\!\sim,\tau_{\textrm{quot}})$ is Hausdorff (Proposition~\ref{proposition_27_dec_2021_16:18}), we get 
  $$
  [\bbz]=[\bbw].
  $$
   As $(\bbz_{h(j)})_{j\in J}$ converges to $\bbw$ in $\overline{\mD}^d$, we have, 
   \phantom{$\overline{\mD}^d_{h(j)})_{j\in J}$}\!\!\!\!\!\!\!\!\!\!\!\!\!\!\!\!\!\!\!\!\!\!\!
   for all $f\in A_{\textrm{sym}}(\mD^d)$, that 
$(f(\bbz_{h(j)}))_{j\in J}$ converges to $f(\bbw)$. 
\phantom{$\overline{\mD}^d_{h(j)})_{j\in J}$}\!\!\!\!\!\!\!\!\!\!\!\!\!\!\!\!\!\!\!\!\!\!\!
Hence $(\varphi_{[\bbz_{h(j)}]} (f))_{j\in J}$ converges to $\varphi_{[\bbw]}f $ for all $f\in A_{\textrm{sym}}(\mD^d)$. 
\phantom{$\overline{\mD}^d_{h(j)})_{j\in J}$}\!\!\!\!\!\!\!\!\!\!\!\!\!\!\!\!\!\!\!\!\!\!\!
Thus 
$(\varphi_{[\bbz_{h(j)}]})_{j\in J}$ in $F$ converges to $\varphi_{[\bbw]}$ in $(M_{A_{\textrm{sym}}(\mD^d)},\tau_{\textrm{Gel}})$. 
\phantom{$\overline{\mD}^d_{h(j)})_{j\in J}$}\!\!\!\!\!\!\!\!\!\!\!\!\!\!\!\!\!\!\!\!\!\!\!
But as $F$ was closed in $\tau_{{\scaleobj{0.75}{\textrm{Gel}}}}$, we have $[\bbw]\in F$. As $[\bbz]=[\bbw]$, we obtain $[\bbz]\in F$. 
\phantom{$\overline{\mD}^d_{h(j)})_{j\in J}$}\!\!\!\!\!\!\!\!\!\!\!\!\!\!\!\!\!\!\!\!\!\!\!
Consequently, $F$ is closed in $\tau_{\textrm{quot}}$ too. 
By Lemma~\ref{lemma_27_dec_2021_15:43}, it follows that $\iota$ is a homeomorphism. 
\phantom{$\overline{\mD}^d_{h(j)})_{j\in J}$}\!\!\!\!\!\!\!\!\!\!\!\!\!\!\!\!\!\!\!\!\!\!\!
 \end{proof}

\section{Contractibility of $M_{A_{\mathrm{sym}}(\mD^d)}$ and its consequences}

\noindent Recall that a topological space $X$ is {\em contractible} if the identity map is null-homotopic, that is, there exists an $x_0\in X$ and a continuous map $H:[0,1]\times X\rightarrow X$ such that 

\smallskip 

\noindent $\;\;\bullet $ $H(0,\cdot)=\textrm{id}_X$ (the identity map 
$\textrm{id}:X\rightarrow X$, $X\owns x\mapsto x$) and 

\smallskip 

\noindent $\;\;\bullet$ $H(1,x)=x_0$ for all $x\in X$ (the constant map $X\owns x\mapsto x_0$). 

\medskip 

\noindent Here $[0,1]\times X$ is given the product topology. 

\begin{theorem}
\label{theorem_28_dec_2021_18:35}
$(M_{A_{\mathrm{sym}}(\mD^d)},\tau_{{\scaleobj{0.75}{\mathrm{Gel}}}})$ is contractible.
\end{theorem}
\begin{proof} It is enough to show that $(\overline{\mD}^d/\!\sim,\tau_{\textrm{quot}})$ is contractible. 
We show that the identity map on $\overline{\mD}^d/\!\sim$ is homotopic to the constant map taking value $[\mathbf{0}]$  everywhere. Define $H:[0,1]\times (\overline{\mD}^d/\!\sim)\rightarrow \overline{\mD}^d/\!\sim$ by 
$$
H(t,[\bbz])=[(1-t)\bbz],\quad \bbz\in \overline{\mD}^d.
$$
The map $H$ is well-defined. Indeed, if $[\bbz]=[\bbw]$ for some $\bbz,\bbw\in \overline{\mD}^d$, 
then $\bbz\sim \bbw$, that is, there exists a $\sigma \in S_d$ such that $\sigma \;\! \bbz =\bbw$, and so    for $t\in [0,1]$, 
$$
\sigma ((1-t) \bbz)=(1-t)\bbw,
$$
 showing that $(1-t)\bbz\sim( (1-t)\bbw)$, that is, 
$[(1-t)\bbz]=[(1-t)\bbw]$. 

\noindent Clearly, $H(0,[\bbz])=[\bbz]$ for all $\bbz\in \overline{\mD}^d$, and so 
$$
H(0,\cdot)=\textrm{id}_{\overline{\mD}^d/\!\sim}.
$$ 
Moreover, $H(1,[\bbz])=[\mathbf{0}]$  for all $\bbz\in \overline{\mD}^d$, and so 
$H(1,\cdot)$ is the constant map on $\overline{\mD}^d/\!\sim$, taking value $[\mathbf{0}]$ everywhere on $\overline{\mD}^d/\!\sim$. 

Finally, we show the continuity of $H$. Let $\bbz\in  \overline{\mD}^d$, $t\in [0,1]$, and 
$\calV$ be a neighbourhood of $[(1-t)\bbz]$. Set 
$$
Z:=\pi^{-1} \calV\subset \overline{\mD}^d.
$$
 Then $Z$ is an open set containing the finite set $\{(1-t)\sigma \bbz:\sigma \in S_d\}$. 
 
 \noindent As the map 
$$
\begin{array}{rccc} 
\Phi: &[0,1]\times \overline{\mD}^d & \rightarrow & \overline{\mD}^d \\[0.1cm]
 &\;\; (t,\bzeta)& \mapsto & (1-t)\bzeta  
 \end{array}
 $$
is continuous, for each $\sigma \in S_d$, there exists an open neighbourhood $I_\sigma$ of $t$ in $[0,1]$, 
and an open neighbourhood $W_\sigma $ of $\bbz$ in $\overline{\mD}^d$ such that 
$$
\Phi(I_\sigma \times W_\sigma)\subset Z.
$$
 Set 
$$
W:=\bigcup_{\sigma \in S_d} W_\sigma, \;\;\textrm{ and }\;\;
I:=\bigcap_{\sigma \in S_d} I_\sigma.
$$ 
Then 
$$
\Phi(I\times W)\subset Z.
$$
 Let 
$$
U=\bigcap_{\sigma \in S_d} \sigma W.
$$
Then $U$ is open in $\overline{\mD}^d$, and for all $\sigma \in S_d$, we have 
$$
\sigma \;\!U=U.
$$ 
Using the fact that for all $\sigma \in S_d$ we have $\sigma \bbz \in W$, it can be seen that $\bbz \in U$. 
Define 
$$
\calU:=  \pi U:=\{[\bzeta]: \bzeta \in U\}.
$$
 Then 
 $$
 \pi^{-1} \calU=U.
 $$
  So $\calU$ is open in $\overline{\mD}^d/\!\sim$. As $\bbz \in U$, we have $[\bbz]\in \calU$. Hence $I\times \calU$ is an open set in $[0,1]\times (\overline{\mD}^d/\!\sim)$ containing $(t,[\bbz])$ such that $H(I\times \calU)\subset \calV$. 
Consequently, $H$ is continuous. \phantom{$\overline{\mD}^d$}
\end{proof}

\subsection{Applications}

\noindent  A commutative unital ring $R$ is {\em projective-free} if every
 finitely generated projective $R$-module is free. 
Recall that if $M$ is a finitely generated
  $R$-module, then

\smallskip 

\noindent $\;\;\bullet$  $M$ is called {\em free} if $M \simeq R^k$ for some integer $k\geq 0$;

 \smallskip 

\noindent $\;\;\bullet$ $M$ is called {\em projective} if there exists an $R$-module $N$
  and an integer $\phantom{\;\;\bullet}$ $ m\geq 0$ such that $M\oplus  N = R^m$.

\medskip 

\noindent 
For $m,n\in \mN$, $R^{m\times n}$ denotes the set of matrices with $m$ rows and $n$ columns having entries from $R$. The identity element in $R^{k\times k}$ having diagonal elements $1$,  and zeroes elsewhere will be denoted by $I_k$. 
In terms of matrices (see \cite[Proposition~2.6]{Coh} or
\cite[Lemma~2.2]{BalRodSpi}), the ring $R$ is projective-free if and
only if every idempotent matrix $P$ is conjugate (by an invertible
matrix $S$) to a diagonal matrix with elements $1$ and $0$ on the diagonal,
that is, for every $m\in \mN$ and every $P\in R^{m\times m}$
satisfying $P^2=P$, there exists an $S\in R^{m\times m}$ such that $S$
is invertible as an element of $R^{m\times m}$, and for some $k\geq 0$, 
$$
S^{-1} P S=\left[\begin{array}{cc}
I_k & 0\\
0 & 0
\end{array}\right].
$$
In 1976, it was shown independently by Quillen and Suslin,  that if $\mF$ is a field, then the polynomial
ring $\mF[x_1, \dots , x_n]$  is
projective-free, settling Serre's conjecture from 1955 (see \cite{Lam}). 
In the context of a commutative semisimple unital complex Banach
algebra $\calA$, \cite[Corollary~1.4]{BS} says that the
contractibility of the maximal ideal space $M_\calA$ in the Gelfand topology, suffices for $\calA$ to be projective-free. Theorem~\ref{theorem_28_dec_2021_18:35} gives: 

\begin{corollary}
$A_{\textrm{\em sym}}(\mD^d)$ is a projective-free ring. 
\end{corollary}

\noindent 
The study of Serre's conjecture and algebraic $K$-theory naturally led to the notion of Hermite rings. A commutative unital ring $R$ is 
{\em Hermite} if 
every finitely generated stably free $R$-module is free. 
(A finitely generated $R$-module $M$  is {\em stably free}  if there exist free finitely generated $R$-modules $F$ and $G$  such that $M \oplus F =G$.)  
 Clearly, every projective-free ring is Hermite. In
terms of matrices, $R$ is Hermite if and only if left-invertible tall matrices over $R$ can be completed to invertible ones (see e.g.  \cite[p.VIII]{Lam}, \cite[p.1029]{Tol}):

\smallskip 

\noindent 
\fbox{\!\!\!
  $\begin{array}{llll} \textrm{for all }k\in \mN \textrm{ and }K\in
    \mN\textrm{ such that }k<K,\textrm{ and}
     \\
     \textrm{for all }f\in R^{K\times k}\textrm{ such that there
     exists a } g\in R^{k\times K} \textrm{ so that } gf=I_k,
     \\
     \textrm{there exists an }f_c\in R^{K\times (K-k)}\textrm{ and there  exists a } G\in R^{K\times K}\\
     \textrm{such that } G\left[\begin{array}{c|c}f &
                                                      f_c\end{array}\right]=I_K.
\end{array}\!\!$}

\smallskip 

\begin{corollary}
$A_{\textrm{\em sym}}(\mD^d)$ is a Hermite ring. 
\end{corollary}

\goodbreak 

\section{$\SL_n(A_{\mathrm{sym}}(\mD^d))=\E_n(A_{\mathrm{sym}}(\mD^d))$}

\noindent 
Let $R$ be a commutative unital ring with multiplicative identity $1$ and additive identity element $0$. Let $n\in \mN$.

 \noindent $\bullet$   Let $\textrm{GL}_n(R)$ denote the {\em general linear group} of all invertible matrices \phantom{$\bullet$} in $R^{\;\!n\times n}$
 (with matrix 
 multiplication). 

\noindent $\bullet$ 
 The {\em special linear group} $\SL_n(R)$ denotes the subgroup of $\textrm{GL}_n(R)$ of 
 \phantom{$\bullet$} all matrices $M$  whose 
  determinant $\det M=1$. 
  
\noindent $\bullet$  An {\em elementary matrix} 
$E_{ij}(\alpha)$ over $R$ has the form $E_{ij}=I_n+\alpha \;\!\bbe_{ij}$, 
 \phantom{$\bullet$} 
 \phantom{$\bullet$} where $i\neq j$, $\alpha \in R$, and $\bbe_{ij}$ is the $n\times n$ matrix whose entry in the 
\phantom{$\bullet$}  
\phantom{$\bullet$} $i$th row and $j$th column is $1$, and all the other entries of $\bbe_{ij}$ are zeros. 

\medskip

\noindent 
 $\E_n(R)$ is the subgroup of $\SL_n(R)$ generated by the elementary 
 matrices. A classical question in algebra is: For all $n\in \mN$, is $\SL_n(R)=\E_n(R)$? 
The answer depends on the ring $R$. For Banach algebras, the following result is known  \cite[\S7]{Mil}:
 
 \begin{proposition}
  \label{prop_11_may_2021_18:36} $\;$
 
 \noindent 
 Let $\calA$ be a commutative unital Banach algebra$,$ $n\in \mN,$ $M\in \eSL_n(\calA)$. Then the following are equivalent:
 \begin{itemize}
 \item $M\in \eE_n (\calA)$.
 \item $M$ is null-homotopic.
 \end{itemize}
 \end{proposition}
 
\begin{definition}[Null-homotopic element of $\SL_n(\calA)$]$\;$

\noindent Let $\calA$ be a commutative unital Banach algebra, and $n\in \mN$. 
An element $M\in \SL_n(\calA)$ is {\em null-homotopic} if $M$ is homotopic to the identity matrix $I_n$, that is there exists a continuous map $H:[0,1]\rightarrow \SL_n(\calA)$ such that 
$H(0)=M$ and $H(1)=I_n$. 
\end{definition} 

\noindent We also elaborate on the Banach algebra structure of $\calA^{n\times n}$ for a Banach algebra $\calA$.
 Let $(\calA,\|\cdot\|)$  be a commutative unital Banach algebra. 
 Then $\calA^{n\times n} $ is a complex algebra with the usual matrix operations. 
 Let $\calA^{n\times 1}$ denote the vector space of all column vectors of size $n$ with entries from $\calA$ 
 and componentwise operations. Then $\calA^{n\times 1}$ is a normed space with the `Euclidean norm' 
 defined by $\displaystyle 
 \|\bbv\|_2^2 :=\|v_1\|^2+\cdots+\|v_n\|^2$ for all $\bbv$ in $\calA^{n\times 1}$, where 
  $\bbv$ has components denoted by $v_1,\cdots,v_n\in \calA$.  If $M\in \calA^{n\times n}$, 
 then the matrix multiplication map, 
  $$
 \calA^{n\times 1}\owns \bbv \mapsto M\bbv \in \calA^{n\times 1}, 
 $$
 is 
 a continuous linear transformation, and we equip $\calA^{n\times n}$ with the induced operator norm, denoted by $\|\cdot\|$ again. Then $\calA^{n\times n}$ with this operator norm is a unital Banach algebra. Subsets of $\calA^{n\times n}$ are given the induced subspace topology. We also state the following observation which will be used later. 
 
 \begin{lemma}
 \label{lemma_14_june_2021_10:24}
 Let $(\calA,\|\cdot\|)$  be a commutative unital Banach algebra, and 
  $M=[m_{ij}]\in \calA^{n\times n}$, where the entry in the $i^{\textrm{\em th}}$ row and $j^{\textrm{\em th}}$ column of $M$ is denoted by $m_{ij}$, $1\leq i,j\leq n$. Then 
 $$
 \|M\|^2\leq \displaystyle \sum_{i=1}^n \sum_{j=1}^n \|m_{ij}\|^2.
 $$
 \end{lemma}
  \begin{proof}  Let $\bbv\in \calA^{n\times 1}$ have components $v_1,\cdots, v_n\in A$. 
  
  \noindent Using the Cauchy-Schwarz inequality in $\mR^n$, we have 
 \begin{eqnarray*}
 \|M\bbv \|_2^2=\sum_{i=1}^n  \Big\| \sum_{j=1}^n m_{ij} v_j\Big\|^2
 \!\!\!\!&\leq&\!\!\!\! \sum_{i=1}^n \Big( \sum_{j=1}^n \|m_{ij}v_j\|\Big)^2
 \leq \sum_{i=1}^n \Big( \sum_{j=1}^n \|m_{ij}\|\|v_j\|\Big)^2
 \\
 \!\!\!\!&\leq&\!\!\!\! \sum_{i=1}^n\sum_{j=1}^n \|m_{ij}\|^2\sum_{k=1}^n \|v_k\|^2
 =
 \sum_{i=1}^n\sum_{j=1}^n \|m_{ij}\|^2 \|\bbv\|_2^2.
 \end{eqnarray*}
 As $\bbv\in \calA^{n\times 1}$ was arbitrary, it follows that $ \|M\|^2\leq \displaystyle \sum_{i=1}^n \sum_{j=1}^n \|m_{ij}\|^2$.
 \end{proof}

\begin{corollary}
 For all $n\in \mN,$ $\eSL_n(A_{\mathrm{sym}}(\mD^d))=\eE_n(A_{\mathrm{sym}}(\mD^d))$. 
\end{corollary}
\begin{proof} We produce the homotopy using dilations (see \cite{Sas}, \cite{DK}). 

\noindent 
We need to show $\SL_n(A_{\textrm{sym}}(\mD^d))\subset \E_n(A_{\textrm{sym}}(\mD^d))$. Let
 $$
 M=\left[\!\!\begin{array}{ccc} f_{11}&\cdots & f_{1n} \\ 
 \vdots & \ddots & \vdots \\
f_{n1} & \cdots & f_{nn} \end{array}\!\! \right]\in \SL_n(A_{\textrm{sym}}(\mD^d)).
$$
 For $t\in [0,1]$, define 
 $$
 H(t)=M_t=\left[\!\!\begin{array}{ccc} f((1-t)\cdot)  &\cdots & f((1-t)\cdot)\\ 
 \vdots & \ddots & \vdots \\
 f_{n1}((1-t)\cdot ) & \cdots & f_{nn}((1-t)\cdot) \end{array} \!\!\right].
 $$
 Then $M_t\in \SL_n(A_{\textrm{sym}}(\mD^d))$ for all $t\in [0,1]$, 
 since $\det (M_t)=\mathbf{1}$. 
 The continuity of the map 
 $$
 [0,1]\owns t\mapsto M_t \in \SL_n(A_{\textrm{sym}}(\mD^d))
 $$
  follows from 
 the compactness of $\overline{\mD}^d$ (making elements of $A_{\textrm{sym}}(\mD^d)$ uniformly continuous), and from Lemma~\ref{lemma_14_june_2021_10:24}.  We note that $M_1= C \mathbf{1}$, 
 where $C\in \SL_n(\mC)$ is the constant matrix given by 
 $$
 C=\left[\!\! \begin{array}{ccc}
 f_{11}(\mathbf{0}) & \cdots &  f_{1n}(\mathbf{0}) \\
 \vdots & \ddots & \vdots \\
   f_{n1}(\mathbf{0}) & \cdots &   f_{nn}(\mathbf{0}) 
 \end{array}\!\!\right] .
 $$
 But $\SL_n(\mC)$ is path-connected, and so there exists a homotopy, say $h:[0,1]\rightarrow \SL_n(\mC)$, 
 taking $C$ to the identity matrix $I_n\in \SL_n(\mC)$. Combining $H$ with $h$, 
 we obtain a homotopy $\widetilde{H}:[0,1]\rightarrow \SL_n(A)$ that takes $M$ to $I_n\in \SL_n(A_{\textrm{sym}}(\mD^d))$:
 $$
 \widetilde{H}(t)=\left\{\begin{array}{ll} H(2t) & \textrm{if }\; t\in \Big[0,\displaystyle
 {\scaleobj{0.75}{\frac{1}{2}}}\Big],\\[0.21cm]
 h(2t-\!1) \mathbf{1} & \textrm{if }\; t\in \Big[{\scaleobj{0.75}{\displaystyle\frac{1}{2}}},1\Big].\end{array}\right.
 $$
 So we have shown that $M\in \SL_n(A)$ is null-homotopic. Consequently, by Propposition~\ref{prop_11_may_2021_18:36},  we have $M\in \E_n(A_{\textrm{sym}}(\mD^d))$.
\end{proof}

\section{Noncoherence of $A_{\mathrm{sym}}(\mD^d)$}

\noindent Let $R$ be  a commutative ring. For $n\in \mN$ and elements $r_1,\cdots,r_n\in  R$, we denote by $\langle r_1,\cdots, r_n \rangle$ the ideal in $R$ generated by $r_1,\cdots, r_n$. The ring $A_{\mathrm{sym}}(\mD^d)$ is not Noetherian, since the
ascending chain condition fails, as demonstrated by the chain of ideals
$$
 \langle B_1(z_1)\cdots B_1(z_d)\rangle 
\subsetneq \langle B_2(z_1)\cdots B_2(z_d)\rangle
\subsetneq \langle B_3(z_1)\cdots B_3(z_d)\rangle
\subsetneq\cdots,
$$ 
where $B_n\in A(\mD)$ are the rational functions defined by  
$$
\displaystyle 
B_n(z):=\prod_{k=1}^n \frac{\alpha_n-z}{1-\alpha_n z}
\;\;\textrm{ and }\;\;
\displaystyle 
\alpha_n =1-\frac{1}{n^2} \;\;(n\in \mN).
$$
In absence of the Noetherian `finiteness' property, the next best finiteness condition in commutative algebra is presented when one has a coherent ring.

\begin{definition}
A unital commutative ring $R$ is said to be {\em coherent} if 

$\bullet$ the intersection 
of any two finitely generated ideals in $R$ is \phantom{finitely }
\phantom{aa$\bullet$} finitely 
generated, and 

 $\bullet$ for
every $a \in  R$, the annihilator $\textrm{Ann}(a) := \{x \in R : ax = 0 \}$ is \phantom{Aaa}
\phantom{aa$\bullet$} finitely generated. 
\end{definition}

\noindent We refer the reader to the article \cite{Gla} for the relevance of the property of
coherence in commutative algebra. The following, which we will need, is a slight abstraction of \cite[Observation~2.1]{MorSas}. 

\begin{theorem}
\label{theorem_6_jan_2021_1053}

\noindent Let $\calR_d$ be the class of all commutative unital subrings of $\mC^{\mD^d}:=\{f:\mD^d\rightarrow \mC\}$ 
under the usual pointwise operations. 

\noindent Let $D:R_d\rightarrow R_1$ and $U:R_1\rightarrow R_d$  be  ring homomorphisms, such that 
$$
DUg=g\textrm{ for all }g\in R_1. 
$$
\noindent Then if $R_d \in \calR_d$ is a coherent ring, then $R_1$ is a coherent ring. 
\end{theorem}
\begin{proof}\footnote{The proof is the same, mutatis mutandis, as that of \cite[Observation~2.1]{MorSas} (since it turns out that there the above abstract properties of $D$ and $U$ were crucial, rather than their concrete explicit form). So we include the details of the short proof here for the convenience of the reader and to keep the discussion self-contained.}
 Suppose that $I$ and $J$ are two finitely generated ideals in $R_1$, say $I =
\langle f_1, \cdots , f_K \rangle$  and $J = \langle g_1, \cdots , g_L\rangle$, and let $I_d$ and $J_d$ 
be the ideals in $R_d$ generated by the functions $U f_k$ ($k = 1 , \cdots , K$), 
 respectively $Ug_{\ell}$ ($\ell=1,\cdots, L$). 
 The coherence of $R_d$ implies that $I_d \cap J_d$ is finitely generated. 
 Let $\{p_1, \cdots , p_M \}$ 
be a set of generators for the ideal $I_d \cap J_d$ in $R_d$. Then 
$\{Dp_1, \cdots , Dp_M\}$ is a set of generators for $I \cap J$. 
Indeed, if $f \in I \cap J$, then in particular,
 $f \in  I = \langle f_1, \cdots , f_K \rangle$, so that there exist 
 $\alpha_1, \cdots , \alpha_K\in  R_1$ such that
$f = \alpha_1 f_1 + \cdots + \alpha_K f_K$. 
Hence we have $U f  = ( U \alpha_1)(U f_1) + \cdots + ( U \alpha_K )(U f_K ) \in  
\langle U f_1, \cdots , U f_K \rangle
= I_d$. Similarly, $U f  \in  J_d$ too. Thus  
$U f  \in I_d \cap J_d = \langle p_1, \cdots , p_M \rangle$, and so
there exist $\gamma_1, \cdots , \gamma_M\! \in\!  R_d$ such that 
$U f  = \gamma_1p_1 + \cdots +\gamma_M p_M$. 
Consequently, $f = DU f = ( D\gamma_1)(Dp_1) + \cdots + ( D\gamma_M )(Dp_M ) 
\in  \langle Dp_1, \cdots , Dp_M\rangle$, and so $I \cap J \subset \langle Dp_1, \cdots , Dp_M \rangle$. 
Vice versa, for any $m = 1 , \cdots , M$, we have 
 $p_m \in I_d = \langle U f_1, \cdots , U f_K \rangle$. So there exist 
 $\theta_1, \cdots , \theta_K \in  R_d$ such that 
 $p_m =\theta_1(U f_1) + \cdots + \theta_K (U f_K )$. 
 Applying $D$, 
 \begin{eqnarray*}
 Dp_m \!\!\!&=&\!\!\! ( D\theta_1)(DU f_1) + \cdots + ( D\theta_K )(DU f_K )
\\
\!\!\!&=&\!\!\! ( D\theta_1)f_1 + \cdots + ( D\theta_K )f_K \in  \langle f_1, \cdots , f_K \rangle = I.
\end{eqnarray*}
 Similarly, each $Dp_m\in J$. Hence $\{Dp_1, \cdots , Dp_M \} 
 \subset  I \cap J$, and so the ideal $\langle Dp_1, \cdots , Dp_M \rangle \subset I \cap J$ too. 
Hence $I \cap J =\langle Dp_1, \cdots , Dp_M \rangle $, and so $ I \cap J$ is finitely generated.

It remains to verify the condition on the annihilators.
Suppose $f \in R_1$. Then $Uf\in R_d$. As $\textrm{Ann}(Uf)$ is finitely generated, 
there exist $h_1,\cdots, h_r\in R_d$ which generate $\textrm{Ann}(Uf)$. 
We claim that $\textrm{Ann}(f )$ is generated by $Dh_1, \cdots , Dh_r$. Suppose $g\in \textrm{Ann}(f )$, that is, $fg=0$. Then $(Uf)(Ug)=U(fg)=U(0)=0$, and so $Ug\in \textrm{Ann}(Uf)=\langle h_1,\cdots, h_r\rangle$. So there exist $\beta_1,\cdots, \beta_r\in R_d$ such that 
$Ug=\beta_1 h_1+\cdots+\beta_r h_r$. Applying $D$,  we obtain 
$
g=DUg=(D\beta_1)(D h_1)+\cdots+(D\beta_r)(D h_r)$ and so $g\in \langle Dh_1,\cdots, Dh_r\rangle$, 
 showing that $\textrm{Ann}(f )\subset \langle Dh_1, \cdots , Dh_r\rangle$. Vice versa, since for $1\leq k\leq r$, we have $h_k\in \textrm{Ann}(Uf)$, it follows that  $(Uf) h_k \! = \! 0$, which yields $f(Dh_k) \! = \! (DUf)(Dh_k) \! = \! D((Uf)h_k) \! = \! D0 \! = \! 0$, that is, $Dh_k\in \textrm{Ann}(f)$. Thus $\langle Dh_1,\cdots, Dh_r\rangle \subset \textrm{Ann}(f)$ too. Consequently,  $\textrm{Ann}(f )$ is finitely generated, completing the proof that  $R_1$ is coherent.
\end{proof}

\noindent In \cite[Observation~2.1]{MorSas}, maps $D:R_d\rightarrow DR_d$ and $U:DR_d\rightarrow R_d$ were defined as follows:
$$
\begin{array}{l}
(Df)(z)=f(z,0,\cdots, 0),  \quad z\in \mD,                \;\;     f\in R_d, \\
(Ug)(z_1,\cdots ,z_d) = g(z_1),            \quad (z_1,\cdots, z_d)\in \mD^d, \;\; g\in DR_d.
\end{array}
$$
However, these cannot be used for $R_d:=A_{\textrm{sym}}(\mD^d)$, 
since the property $UDf\in R_d$ for all $f\in R_d$, demanded in \cite[Observation~2.1]{MorSas},  fails: For example,  for $d\!=\!2$ and with   $f\in R_2 \! = \! A_{\mathrm{sym}}(\mD^2)$ given by 
$ f(z,w) \! = \! z+w$ for $(z,w)\in \mD^2$, we get $(UD f)(z,w) \!= \!z$, so that $UDf\not\in A_{\textrm{sym}}(\mD^2) \!= \!R_2$. 

\noindent In order to apply Theorem~\ref{theorem_6_jan_2021_1053}, we instead use the following.

\begin{lemma}
\label{lemma_6_jan_2022_17:53}
Let $R_d:=A_{\mathrm{sym}}(\mD^d)$, and define the maps 
$D:R_d\rightarrow R_1$ and 
$U:R_1\rightarrow R_d$ by 
$$
\begin{array}{l}
(Df)(z)=f(z,\cdots, z),  \quad z\in \mD,                \;\;         f\in R_d, \\[0.1cm]
(Ug)(z_1,\cdots ,z_d) = \displaystyle g\Big(\frac{z_1+\cdots+z_d}{d}\Big),           
 \quad (z_1,\cdots, z_d)\in \mD^d, \;\; g\in R_1.
\end{array}
$$
Then $D,U$ are well-defined ring homomorphisms$,$ and 
$$
DUg=g\textrm{ for all }g\in R_1.
$$ 
\end{lemma}
\begin{proof} For $f\in R_d$, $Df$ is clearly holomorphic in $\mD$, and possesses a continuous extension to $\overline{\mD}$. Thus $Df\in A(\mD)=R_1$. Also, it is clear that $D$ is a ring homomorphism. 

The map $\mD^d\owns (z_1,\cdots, z_d)\mapsto \frac{z_1+\cdots+z_d}{d}\in \mD$ is continuous. 
So it follows that $Ug \in A(\mD^d)$. 
For $g\in R_1$, $\sigma \in S_d$, and $(z_1,\cdots, z_d)\in \mD^d$,  
\begin{eqnarray*}
(Ug)(z_{\sigma(1)},\cdots, z_{\sigma(d)})
\!\!\!&=&\!\!\!
g\Big(\frac{z_{\sigma(1)}+\cdots+ z_{\sigma(d)}}{d}\Big)=g\Big(\frac{z_1+\cdots+z_d}{d}\Big)\\
\!\!\!&=&\!\!\!(Ug)(z_1,\cdots ,z_d),
\end{eqnarray*}
showing that $Ug\in A_{\mathrm{sym}}(\mD^d)=R_d$. Moreover, it is clear that $U$ is a ring homomorphism. 

Finally, we check that $DUg=g$ for all $g\in R_1=A(\mD)$.  Let 
$g\in A(\mD)$. Then for all $(z_1,\cdots, z_d)\in \mD^d$, we have 

$$
\Phi(z_1,\cdots, z_d):=(Ug)(z_1,\cdots, z_d)=g\Big(\frac{z_1+\cdots+ z_d}{d}\Big).
$$
 Hence for all $z\in \mD$, 
$$
(D(Ug))(z)= (D\Phi)(z)=\Phi(z,\cdots, z)
=g\Big(\frac{z+\cdots+z}{d}\Big)
=g(z).
$$
This completes the proof.
\end{proof}

\noindent As the disk algebra $R_1:=A(\mD)$ is not coherent (see \cite{McVRub}, or \cite{MorSas} for a simpler proof), it follows from Theorem~\ref{theorem_6_jan_2021_1053} and Lemma~\ref{lemma_6_jan_2022_17:53} that:

\begin{corollary}
 $A_{\mathrm{sym}}(\mD^d)$ is not coherent. 
\end{corollary}

\section{The Wiener algebra $W^+_{\mathrm{sym}}(\mD^d)$ of symmetric functions}

\noindent Analogous to the polydisc algebra of symmetric functions, one can also consider $W^+_{\mathrm{sym}}(\mD^d)$, the analytic Wiener algebra on the polydisc, consisting of symmetric functions. All the results of the previous sections can be proved in the same way (mutatis mutandis) also for this algebra. We discuss this algebra here for a different reason. A fundamental theorem in algebra states that every symmetric polynomial can be expressed uniquely as a polynomial in the symmetric polynomials. One may ask if an analogue of this result carries over for the $W^+_{\mathrm{sym}}(\mD^d)$, and 
we investigate this question here. We begin with the pertinent definitions.

\noindent Let 
 $$
\mN_0:=\{0,1,2,3,\cdots\}.
$$ 
 For $d\in \mN$, $\bbn =(n_1,\cdots, n_d)\in \mN_0^d$, and 
$\bbz=(z_1,\cdots, z_d)\in \mC^d$, we set 
\begin{eqnarray*}
|\bbn|\!\!\!&:=&\!\!\!n_1+\cdots+n_d,\;\;\textrm{ and }\\
\bbz^{\bbn}\!\!\!&:=&\!\!\!z_1^{n_1}\cdots z_d^{n_d}, \;\; z_k^{0}:=1 \;(1\leq k\leq d).
\end{eqnarray*}
The classical {\em analytic Wiener polydisc algebra} $W^+(\mD^d)$ is the Banach algebra 
$$
W^+(\mD^d)=\Big\{ f:\mD^d\rightarrow \mC: f(\bbz)=\sum_{\bbn \in \mN_0^d} c_{\bbn} \bbz^\bbn \; 
(\bbz\in \mD^d), \; \sum_{\bbn \in \mN_0^d} |c_{\bbn}|<\infty\Big\},
$$
with pointwise operations, and the norm $\|\cdot\|_1$ given by 
$$
\|f\|_1:=\sum_{\bbn \in \mN_0^d} |c_{\bbn}|,\textrm{ for }f=\sum_{\bbn \in \mN_0^d} c_{\bbn} \bbz^\bbn\in W^+(\mD^d).
$$
We define the {\em analytic Wiener algebra of symmetric functions}, $W^+_{\textrm{sym}}(\mD^d)$, 
to be the subalgebra of $W^+(\mD^d)$ consisting of symmetric functions:
$$
W^+_{\textrm{sym}}(\mD^d):=
\{f\in W^+(\mD^d): f\textrm{ is a symmetric function}\}.
$$
Then $W^+_{\textrm{sym}}(\mD^d)$ is a Banach subalgebra of $W^+(\mD^d)$ with the same norm $\|\cdot\|_1$. We also have $W^+_{\textrm{sym}}(\mD^d) \subset A_{\textrm{sym}}(\mD^d)$, and  for all $f\in W^+_{\textrm{sym}}(\mD^d)$, 
$$
\|f\|_\infty\leq \|f\|_1
$$
A polynomial $p\in \mC[z_1,\cdots, z_d]$ is {\em symmetric} if 
$$
p(z_1,\cdots, z_d)=p(z_{{\scaleobj{0.75}{\sigma(1)}}},\cdots, z_{{\scaleobj{0.75}{\sigma(d)}}})\textrm{ for all }\sigma \in S_d.
$$

\begin{proposition}
\label{prop_29_dec_2021_17:39}$\;$

\noindent 
The set of symmetric polynomials is dense in $W^+_{\mathrm{sym}}(\mD^d)$. 
\end{proposition}
\begin{proof} As polynomial functions are dense in $W^+(\mD^d)$, we know that given any $f\in W^+_{\textrm{sym}}(\mD^d)\subset W^+(\mD^d)$, there is a sequence of polynomials $(p_n)_{n\in \mN}$ that converges to $f$ in the $\|\cdot\|_1$-norm. But then it follows that 
$(\sigma \;\!p_n)_{n\in \mN}$ conerges to $\sigma f=f$ in the $\|\cdot\|_1$-norm. Hence 
$$
q_n:=\frac{1}{d!} \sum_{\sigma \in S_d} \sigma \;\!p_n\rightarrow \frac{1}{d!}\sum_{\sigma \in S_d} f=f 
$$
as $n\rightarrow \infty$ in the $\|\cdot\|_1$-norm. As $(q_n)_{n\in \mN}$ is a sequence of symmetric polynomials, the proof is complete. 
\end{proof}

\noindent In fact, we can show the  same result in a more constructive manner as follows. 
Let $f\in W^+_{\mathrm{sym}}(\mD^d)$ have the Taylor expansion
$$
f=\sum_{\bbn\in \mN_0^d} c_{\bbn} \bbz^\bbn\quad (\bbz\in \mD^d).
$$
Then defining the polynomials 
$$
p_k=\sum_{|\bbn|=k}c_{\bbn} \bbz^\bbn \quad (k\geq 0),
$$
we have that $p_k$ is a symmetric polynomial (shown below) of (total) degree $k$, and 
\begin{equation}
\label{eq_29_dec_2021_2:05}
f=\sum_{k=0}^\infty p_k. \quad \;\;\;\phantom{(k\geq 0)}
\end{equation}
Thus the sequence of symmetric polynomials $(p_0+p_1+\cdots+p_n)_{n\in \mN}$ converges to $f$ in $W^+_{\mathrm{sym}}(\mD^d)$. To see that $p_k$ is a symmetric polynomial, firstly 
$$
c_\bbn=\frac{1}{n_1 ! \cdots n_d!} \frac{\partial^{|\bbn|} f\hfill}{\partial z_1^{n_1}\cdots z_d^{n_d}\hfill}(\mathbf{0}),
$$
and since $f$ is symmetric, it follows that for any $\sigma \in S_d$, 
$$
\frac{\partial^{|\bbn|} f\hfill}{\partial z_1^{n_1}\cdots z_d^{n_d}\hfill}(\mathbf{0})
=
\frac{\partial^{|\bbn|} f\hfill}{\partial z_1^{n_{{\scaleobj{0.81}{\sigma^{{\scaleobj{0.81}{-1}}}(1)}}}}\cdots z_d^{n_{{\scaleobj{0.81}{\sigma^{{\scaleobj{0.81}{-1}}}(d)}}}}\hfill}(\mathbf{0}).
$$
Consequently, 
\begin{eqnarray*}
\!\!\!&&\!\!\!(\sigma\;\! p_k)(z_1,\cdots, z_d)=p_k(z_{{\scaleobj{0.81}{\sigma(1)}}},\cdots, z_{{\scaleobj{0.81}{\sigma(d)}}})\phantom{\sum}
\\
\!\!\!&=&\!\!\!
\sum_{|\bbn|=k} 
\frac{1}{n_1 ! \cdots n_d!} \frac{\partial^{|\bbn|} f\hfill}{\partial z_1^{n_1}\cdots z_d^{n_d}\hfill}(\mathbf{0})
z_{{\scaleobj{0.81}{\sigma(1)}}}^{n_1}\cdots z_{{\scaleobj{0.81}{\sigma(d)}}}^{n_d}\\
\!\!\!&=&\!\!\!
\sum_{|\bbn|=k} 
\frac{1}{{n_{{\scaleobj{0.81}{\sigma^{{\scaleobj{0.81}{-1}}}(1)}}}} ! \cdots {n_{{\scaleobj{0.81}{\sigma^{{\scaleobj{0.81}{-1}}}(d)}}}}!} \frac{\partial^{|\bbn|} f\hfill}{\partial z_1^{n_{{\scaleobj{0.81}{\sigma^{{\scaleobj{0.81}{-1}}}(1)}}}}\cdots z_d^{n_{{\scaleobj{0.81}{\sigma^{{\scaleobj{0.81}{-1}}}(d)}}}}\hfill}(\mathbf{0})
z_{1}^{n_{{\scaleobj{0.81}{\sigma^{{\scaleobj{0.81}{-1}}}(1)}}}}\cdots z_{d}^{n_{{\scaleobj{0.81}{\sigma^{{\scaleobj{0.81}{-1}}}(d)}}}}
\\[0.15cm]
\!\!\!&=&\!\!\!p(z_1,\cdots, z_d).
\end{eqnarray*}
 Let $d\in \mN$, and $n\in \mN$ is such that $n\leq d$. 
An {\em elementary symmetric polynomial of degree $n$ in $d$ variables} is a symmetric polynomial $s_n$ of the form  
$$
s_n = \sum_{1\leq i_1<i_2<\cdots<i_n \leq d} z_{i_1} \cdots z_{i_d}.
$$

\goodbreak

\noindent 
For example if $d=2$, and we call the two variables $z,w$, then 
\begin{eqnarray*}
s_1\!\!\!&=&\!\!\!z+w, \\
s_2\!\!\!&=&\!\!\!zw; \phantom{AAAAAAAAAAA}
\end{eqnarray*} 
if $d=3$, then 
\begin{eqnarray*}
s_1\!\!\!&=&\!\!\!z_1+z_2+z_3,\\
s_2\!\!\!&=&\!\!\! z_1z_2+z_1z_3+ z_2z_3, \textrm{ and}\\
s_3\!\!\!&=&\!\!\!z_1z_2z_3;
\end{eqnarray*}
 and so on. 
 
 A fundamental result on symmetric polynomials states that every symmetric polynomial can be expressed uniquely as a polynomial in the symmetric polynomials (see e.g. \cite[Chap.14, Thm. 3.4]{Art}):

\begin{proposition}
\label{prop_29_dec_2021_15:35}
Every symmetric polynomial $p\in \mC[z_1,\cdots, z_d]$ can be written in a unique way as a polynomial in the elementary symmetric polynomials, that is, there exists a unique polynomial $q\in \mC[s_1,\cdots, s_d]$ such that 
 $
p(z_1,\cdots, z_d)= q\;\!\pmb{(}s_1(z_1,\cdots, z_d), \cdots, s_d(z_1,\cdots, z_d)\pmb{)}.
$
\end{proposition}

\noindent For example, if $d=2$, then 
$$
z^2+w^2= (z+w)^2-2zw=s_1^2-2s_2.
$$  
An immediate corollary of this result and of Proposition~\ref{prop_29_dec_2021_17:39} is the following.

\begin{corollary} 
The set 
$\mC[s_1,\cdots, s_d]$ of polynomials in the elementary symmetric polynomials is dense in $W^+_{\mathrm{sym}}(\mD^d)$.
\end{corollary}

\subsection{An open question} 

\noindent For $f\in W^+_{\mathrm{sym}}(\mD^d)$, by using the decomposition \eqref{eq_29_dec_2021_2:05} above, and expressing each $p_k$ in terms of a polynomial in $s_1,\cdots, s_d$, one obtains a formal power series $g\in \mC\llbracket s_1,\cdots, s_d\rrbracket$. 
One could hope that $g$ then converges on a domain containing the set 
$$
\{(\zeta_1,\cdots,\zeta_d)\in \mC^d:\exists \bbz\in \mD^d\textrm{ such that } s_1(\bbz)=\zeta_1,\;\cdots, s_d(\bbz)=\zeta_d\},
$$
 so that  
$$
f(\bbz)=g(s_1(\bbz),\cdots, s_d(\bbz)) \quad (\bbz\in \mD^d)
$$
could hold.  Is this true? 

\goodbreak 

\noindent For example, let $d=2$, and consider $f$ given by 
$$
f=\sum_{n=1}^\infty \frac{1}{n^2 2^n} (z^2+w^2)^n.
$$
As each term in the summand $(z^2+w^2)^n$ has total degree $2n$, we have 
\begin{eqnarray*}
\|(z^2+w^2)^n\|_1\!\!\!&=&\!\!\!\binom{n}{0}+\cdots +\binom{n}{n}\\
\!\!\!&=&\!\!\!(1+1)^n=2^n, 
\end{eqnarray*}
and so   
$$
\|f\|_1= \sum_{n=1}^\infty \frac{1}{n^2 2^n} 2^n=\sum_{n=1}^\infty \frac{1}{n^2 }<\infty.
$$
Hence $f\in  W^+_{\textrm{sym}}(\mD^d)$. Using 
$$
z^2+w^2=(z+w)^2-2zw=s_1^2-2s_2,
$$ 
we obtain for $(z,w)\in \mD^2$ that 
\begin{eqnarray*}
f(z,w)\!\!\!&=&\!\!\!\sum_{n=1}^\infty \frac{1}{n^2 2^n} (z^2+w^2)^n\\
\!\!\!&=&\!\!\!\sum_{n=1}^\infty \frac{1}{n^2 2^n} ((s_1(z,w))^2-2s_2(z,w))^n
\\
\!\!\!&=&\!\!\!\sum_{n=1}^\infty \sum_{k=0}^n \frac{1}{n^2 2^n}\binom{n}{k} (-2)^k (s_1(z,w))^{2n-2k} (s_2(z,w))^{k}. 
\end{eqnarray*}
One would then like to conclude that for $(z,w)\in \mD^2$,
\begin{eqnarray*}
f(z,w)\!\!\!&=&\!\!\! g(s_1(z,w), s_2(z,w)),
\end{eqnarray*}
where $g\in \mC\llbracket s_1,s_2 \rrbracket$ is given by 
\begin{eqnarray*}
g(s_1,s_2)\!\!\!&=&\!\!\!
 \sum_{n=1}^\infty \sum_{k=0}^n \frac{1}{n^2 2^n}\binom{n}{k} (-2)^k s_1^{2n-2k} s_2^k
\\
\!\!\!&=&\!\!\!
\frac{s_1^2-2s_2}{1^2 2^1}+\frac{(s_1^2-2s_2)^2}{2^2 2^2}+\frac{(s_1^2-2s_2)^3}{3^2 2^3}+\cdots \\
\!\!\!&=&\!\!\!
-s_2+\frac{1}{2}s_1^2+\frac{1}{4} s_2^2-\frac{1}{9} s_2^3-\frac{1}{2}s_1^2s_2+\cdots.
\end{eqnarray*}

%
%\bigskip
%
%\noindent 
%{\bf Data Availability Statement:} The manuscript has no associated data.
%

\end{document}